\documentclass[12pt,twoside]{article} 
\usepackage{amsmath, amsthm,amssymb, amsfonts, bm} 
\usepackage{authblk} 
\usepackage{hyperref} 
\usepackage{hyperref}
\usepackage{fancyhdr}
\usepackage[a4paper,margin=1in,includeheadfoot]{geometry}

\hypersetup{
    pdftitle={The Resolvent Mean and The Parametrized $\mathcal{A} \sharp \mathcal{B}$}, 
    pdfauthor={Alemeh Sheikhhosseini, Eman Aldabbas, Mohammad Sababheh },
    colorlinks=true,
    linkcolor=blue,
    urlcolor=blue
}
\newtheorem{theorem}{Theorem}[section]
\newtheorem{lemma}[theorem]{Lemma}
\newtheorem{definition}{Definition}[section]
\newtheorem{remark}[theorem]{Remark}
\newtheorem{example}[theorem]{Example}
\newtheorem{corollary}[theorem]{Corollary}
\newtheorem{proposition}[theorem]{Proposition}
\numberwithin{equation}{section}

\providecommand{\keywords}[1]{\small \textbf{\textit{Keywords---}} #1}
\providecommand{\MSC}[1]{\small \textbf{\textit{MSC Classification 2020---}} #1}

\title{The Resolvent Mean and The Parametrized $\mathcal{A} \sharp \mathcal{B}$}

\author[1]{Alemeh Sheikhhosseini\thanks{\href{mailto:sheikhhosseini@uk.ac.ir}{sheikhhosseini@uk.ac.ir}}}
\author[2]{Eman Aldabbas \thanks{\href{e\_aldabbas@ju.edu.jo}{e\_aldabbas@ju.edu.jo, aldabbas@ualberta.ca}}}
\author[3]{Mohammad Sababheh \thanks{\href{sababheh@psut.edu.jo}{ sababheh@yahoo.com}}}
\affil[1]{Department of Pure Mathematics, Faculty of Mathematics and Computer, Shahid Bahonar University of Kerman, Kerma 100190, Iran}
\affil[2]{Department of Mathematics, University of Jordan, Amman 11942, Jordan}
\affil[3]{Department of Basic Sciences, Princess Sumaya University for Technology\\ Amman 11941, Jordan}

\date{} 

\begin{document}

\maketitle

\abstract{Resolvent average and weighted \(\mathcal{A}\sharp \mathcal{H}\)-mean have been defined recently for positive definite matrices. Since the class of accretive matrices provides a general framework for addressing certain known results on positive matrices, this paper extends the notions of resolvent average and the weighted \(\mathcal{A}\sharp \mathcal{H}\)-mean to accretive matrices and discusses some of their properties.\\
The obtained results happen to be legitimate generalizations of those known results on positive definite matrices.\\
Among many results, we show that if $A,B$ are positive definite matrices, and $0\leq\lambda\leq 1, \mu>0$, then
\[\mathcal{R}_{\mu}(A,B,1-\lambda,\lambda)+\mu I \geq C \Big(A \sharp_{\bm{\lambda}} B +\mu I\Big),\]
where $\mathcal{R}_{\lambda}$ is the resolvent average,  $\sharp_{\bm{\lambda}}$ is the weighted geometric mean and $I$ is the identity matrix, for some positive constant  $C$; as a new relation between the resolvent average and the geometric mean.}




\MSC{47A64, 15B48,15A24, 47A63,47B44}

\maketitle

\section{Introduction }
Let \(\mathcal{M}_n\) denote the algebra of \(n\times n\) complex matrices, with identity matrix \(I\) and zero matrix $O$. A   matrix \(A\in \mathcal{M}_n\) is called positive semi-definite, and is denoted as \(A\geq O\), if \(\langle Ax, x\rangle \geqslant0 \) for all \(x \in \mathbb{C}^n\). If \(\langle Ax, x\rangle > 0 \) for all non-zero \(x \in \mathbb{C}^n\), then \(A\) is called positive definite and is denoted as
\(A>O\).

The class of all positive definite matrices will be denoted by \(\mathcal{M}_n^{++}\). For two Hermitian matrices \(A,B\), we write that \(A\geq B\) (or \( A > B\)) if \( A-B \geq O\) (or \(A-B>O\)). In particular, If $M\in\mathbb{R}$, we write \(A\leq M\) to mean that \(MI -A\geq O\). \\
A linear map $\Phi : \mathcal{M}_n \rightarrow \mathcal{M}_n$ is called positive if
$\Phi(A)\geq O$ when $A\geq O$. If $\Phi(I)=I$, then $\Phi$ is called unital.\\
If \(A\in \mathcal{M}_n\), its numerical range is denoted by \(\mathcal{W}(A)\), and is defined as  \[\mathcal{W}(A)= \{\langle Ax,x \rangle: x\in \mathbb{C}^n \text{ with }\| x\|=1\},\]
which is a convex and compact subset of \(\mathbb{C}\) that contains the spectrum of \(A\), see e.g.\cite{Ball,Davis,Parker}. 
The Cartesian decomposition of \(A\in \mathcal{M}_n\) is \(A= \Re A + i \Im A\), where \(\Re A= \frac{A+A^*}{2},\Im A= \frac{A-A^*}{2 i}\) are the real and imaginary parts of $A$ respectively, and \(A^*\) is the adjoint matrix of \(A\). A matrix \(A\in \mathcal{M}_n\) is called accretive if \(\Re A\in \mathcal{M}_n^{++}\). This condition is equivalent to the fact that the numerical range of \(A\in \mathcal{M}_n\) satisfies \[\mathcal{W}(A)\subset   \{z\in\mathbb{C}: \Re z\geq 0\}.\]
For simplicity, we will write \(A\in \Gamma_n\) to mean that \(A\in\mathcal{M}_n\) is accretive. Clearly, the class \(\Gamma_n\) contains the class \(\mathcal{M}_n^{++}\).\\
As an alternative terminology, the term sectorial matrices is often used in reference to accretive matrices. If, for \(\alpha\in [0,\pi/2)\), we set
\[\mathcal{S}_{\alpha}= \{ z\in \mathbb{C}: \Re(z)\geq 0 \text{ and } \vert \Im(z)\vert \leq \tan(\alpha) \Re(z)\},\]
then, clearly, \(\mathcal{S}_{\alpha}\) is a sector in the right half-plane. A matrix \(A\in \mathcal{M}_n\) is called sectorial if \(\mathcal{W}(A)\subset \mathcal{S}_{\alpha}\) for some \(\alpha\in [0,\pi/2)\). The class of all matrices in \(\mathcal{M}_n\) with numerical range in \(\mathcal{S}_{\alpha}\) will be denoted by \(\Pi_{n,\alpha}\).\\
One can show that \[ A\in \Gamma_n \Leftrightarrow A\in \Pi_{n,\alpha}\, \text{ for some } \alpha\in[0,\pi/2).\]
It is therefore clear that  \[\Gamma_n= \displaystyle \bigcup_{\alpha\in [0,\pi/2)}\, \Pi_{n,\alpha}. \]
Accretive matrices exhibits numerous interesting properties and have recently attracted considerable interest. For more on accretive matrices, we refer the reader to \cite{Drury2014,Drury,bedrani2021positive,bedrani2021weighted,bedrani2021numerical,Furuichi}.\\
Throughout this paper, for a positive unital linear map \(\Phi (\cdot) \) and \( \textbf{A}= (A_1, A_2\cdots, A_m)\), we define \( \Phi(\textbf{A})= (\Phi(A_1), \Phi(A_2), \cdots, \Phi (A_m) ) \) and  \( \Re \textbf{A} = (\Re A_1, \Re A_2,\cdots,  \Re A_m )\). By  \textbf{\(\bm{\bm{\bm{\lambda}}}\)} we mean the \(m\)-vector \( ({\lambda}_1,{\lambda}_2,\cdots,{\lambda}_m)\) where  \({\lambda}_i\in [0,1]\) such that \(\displaystyle \sum_{i=1}^m\, \lambda_i=1\).

A binary operation \(\sigma\) on the class of positive definite matrices is called a connection if the following conditions are satisfied:\begin{itemize}
    \item[1.] \(A\leq C \text{ and } B\leq D \text{, then } A \sigma B \leq C\sigma D\).
    \item[2.] \(C (A \sigma B) C \leq (CAC)\sigma (CBC)\) for any \(C\in \mathcal{M}_n\).
    \item[3.]  \(A_n \downarrow A \text{ and } B_n\downarrow B\), then \((A_n \sigma B_n) \downarrow A \sigma B\). 
\end{itemize}
An operator mean \(\sigma\) in the sense of Kubo-Ando \cite{kubo}  is a normalized connection (i.e, \(I\sigma I = I\)), and is defined as
\begin{equation}\label{matrix mean via f}
   A\sigma B= A^{1/2} f\left(A^{-1/2} B A^{-1/2}\right) A^{1/2}, 
\end{equation}
where \(f: (0,\infty) \rightarrow (0,\infty)\) is an operator monotone function with \(f(1)=1\). In this case, \(f\) is called the representing function of the matrix mean \(\sigma\). 
Recently, Bedrani et al. \cite{bedrani2021positive} showed that \eqref{matrix mean via f} holds for accretive matrices. Among the most well-established means for \( \textbf{A}= (A_1, A_2\cdots, A_m)\)   are the weighted arithmetic mean \(\mathcal{A}_{\bm{\lambda}}\), and the weighted harmonic mean \(\mathcal{H}_{\bm{\lambda}}\) which are defined for \(A_i\in \Gamma_n, \, i=1,2,\cdots,m\),  by
\begin{equation}\label{e01}
\mathcal{H}_{\bm{\lambda}}(\textbf{A})= \left(  \sum _{i=1}^{m}\lambda_i A_{i}^{-1}  \right)^{-1},
\end{equation}
and
\begin{equation}\label{e001}
\mathcal{ A}_{\bm{\lambda}}(\textbf{A} )=  \sum _{i=1}^{m}\lambda_i A_{i}. 
\end{equation}
On the other hand,  the weighted geometric mean \(A\sharp_{\lambda} B\)  of \(A,B \in \Gamma_n\) is given by 
\[A\sharp_{\lambda} B = A^{1/2} \left( A^{-1/2} BA^{-1/2} \right)^{\lambda} A^{1/2},\; \lambda\in [0,1].\]
If \(\bm{\lambda}=(1/2,1/2)\), we simply write \(\mathcal{A}(A,B)\), and \(\mathcal{H}(A,B)\) instead of \(\mathcal{A}_{\bm{\lambda}}(A,B)\) and \( \mathcal{H}_{\bm{\lambda}}(A,B)\). Also, if $\lambda=\frac{1}{2}$, we write \( A\sharp B\)  instead of \(A \sharp_{1/2} B\).\\
The geometric mean was extended to positive semi-definite matrices by Ando \cite{Ando}, and subsequent research has established its significance in matrix inequalities, semi-definite programming (scaling point \cite{Hauser,Nesterov}), and geometry (geodesic midpoint \cite{Bhatia03exp,Bhatia06geo}). One of the key properties of the geometric mean of two positive definite  matrices is that it is defined by a Riccati matrix equation. For any two positive definite matrices \(A,B\), the geometric mean \(A \sharp B\) is the unique positive definite solution of the Riccati matrix equation \(XA^{-1}X=B\). \\
As a result, the geometric mean of \(A,B\) is the metric midpoint of the of arithmetic mean \(\mathcal{A}(A,B)\) and the harmonic mean \(\mathcal{H}(A,B)\) for both the trace metric and the Thompson metric, see \cite{Law}. That is, 
\begin{equation}\label{Intro-G(A,H)=G(A,B)}
    A\sharp B= \mathcal{A}(A,B)\sharp \mathcal{H}(A,B).
\end{equation}
In \cite{Bau1}, Bauschke et al. introduced a very interesting and new notion of proximal average within the field of convex analysis. Associating a positive semi-definite matrix \(A\) with the quadratic convex function \(q_A(x)= \langle Ax,x\rangle\) enables this new average to be extended to positive semi-definite matrices under the name of resolvent average. 

For \(\textbf{A}=(A_1,A_2,\cdots,A_m)\), the (parameterized) resolvent average of $ \textbf{A} $ with weight  $ \bm{\lambda} $ is defined as follows
\[
\mathcal{R} _{\mu}(\textbf{A},\bm{\lambda})= \left(  \sum _{i=1}^{m}\lambda_{i} (A_{i}+ \mu I)^{-1} \right)^{-1}-\mu I, \quad \mu \geq 0,
\]
and for \(\mu = \infty, \mathcal{R} _{\infty}(\textbf{A},\bm{\lambda})= \mathcal{A}_{\bm{\lambda}}(\textbf{A})\).\\
This is motivated by the fact that when $ \mu=1 $\\
$$\left( \mathcal{R} _{1}(A,\bm{\lambda})+I\right)^{-1}= \sum_{i=1}^{m} \lambda_{i}(A_{i}+I)^{-1}, $$
which says that the resolvent of $ \mathcal{R} _{1}(A, \bm{\lambda} ) $ is the arithmetic mean of resolvent of $ A_{1}, \ldots, A_{m} $ with weight $\bm{\lambda}$. \\
Basic properties of the resolvent average have been established. For instance, the resolvent average interpolates between the weighted harmonic and the arithmetic means, which are attained by appropriate limits. For more on the resolvent average we refer the reader to \cite{Bau1,Bau,ku1,ku2}.\\ Closely related mean, called the weighted \(\mathcal{A}\sharp \mathcal{H}\)-mean, has  been introduced in \cite{kim}. Inspired by \eqref{Intro-G(A,H)=G(A,B)}, Kim et al. introduced the geometric mean of weighted \(m\)-variable arithmetic and harmonic means, and called it the weighted \(\mathcal{A}\sharp \mathcal{H}\)-mean. This is defined as   
 
\[\mathcal{L} _{\mu}( \textbf{A},\bm{\lambda})=  \left( \sum_{i=1}^{m} \lambda_{i}(A_{i}+\mu I)\right) \sharp \left( \sum_{i=1}^{m} \lambda_{i}(A_{i}+\mu I)^{-1}\right)^{-1}-\mu I,\, \mu\geq 0,\]
and   
    \[\mathcal{L} _{\mu}( \textbf{A},\bm{\lambda})= \mathcal{L}^{-1} _{-\mu}( \textbf{A}^{-1},\bm{\lambda}),\,\mu<0.\]
Just like the resolvent average, the parametrized weighted \(\mathcal{A}\sharp \mathcal{H}\)-mean interpolates between the weighted harmonic and the arithmetic means and can be recovered by iteration of the arithmetic mean and the resolvent average. \\

The purpose of this paper is to extend the notion of the parametrized resolvent average and the notion of the parameterized \(\mathcal{A}\sharp \mathcal{H}\)-mean from the context of positive definite matrices to accretive ones, and to establish a relation between the resolvent average and the geometric mean of two positive definite matrices. \\

\section{Preliminaries}
In this section, we outline various results needed for the subsequent discussion, all of which are available in the referenced sources. We start this section with findings on matrix means for accretive and positive definite matrices. Then definitions of the parametrized resolvent average and the parametrized weighted \(\mathcal{A}\sharp \mathcal{H}\)-mean for positive definite matrices and their properties are listed.\\

Let \(A,B \in \Gamma_n\). Then the weighted geometric mean of \(A,B\) with weight \(\lambda\in[0,1]\), denoted by \(A\sharp_{\lambda} B\), is defined as follows:
\[A\sharp_{\lambda} B = A^{1/2}\left(A^{-1/2} B A^{-/12} \right)^{\lambda} A^{1/2}.\]
\begin{lemma}\label{Propertiesof the weighted geom mean}\cite{Rai,bedrani2021positive} 
Let \(A,B\in \Gamma_n\) and let \(\lambda\in [0,1]\). Then we have
   \begin{itemize}
       \item[1.] \(A\sharp_{\lambda} A = A\).
       \item[2.] \(A\sharp B= B\sharp A\).
       \item[3.] \(\left(A \sharp_{\lambda} B\right)^{-1}= A^{-1}\sharp_{\lambda} B^{-1}\).
       \item[4.] \((h A)\sharp_{\lambda} (k B) = h^{1-\lambda} k^{\lambda}(A\sharp_{\lambda} B) \).
       \item[5.] If \(A,B,C,D \in M_n^{++}\) such that \(A\leq C, B\leq D\), then \(A\sharp_{\lambda} B \leq C\sharp_{\lambda} D\). 
   \end{itemize} 
\end{lemma}
\begin{lemma}\label{G(A,H)=G(A,B)}\cite{Law}
Let \(A,B\in \mathcal{M}_n^{++}\). Then
\[(A+B)\sharp (A^{-1}+B^{-1})^{-1}= A \sharp B.\]
\end{lemma}

\begin{lemma}[Ando's inequality]\label{l-h}\cite[Theorem 5.8, p. 145]{pec}
 Let \( A, B > O\) and \(\sigma\) be a matrix mean. Then if \(\Phi \) is a positive unital linear map, then
 
\[\Phi (   A \sigma B )  \leq    \Phi(A)\sigma\, \Phi(B).\]
 \end{lemma}

\begin{lemma}\label{l3}\cite{Bha2}
Let $ \Phi $ be a unital positive linear map. Then for $ A \in \mathcal{M}_n^{++} $\\
\begin{equation}\label{e2}
\Phi(A)^{-1} \leq \Phi (A^{-1}).
\end{equation}
\end{lemma}

\begin{lemma}\label{l4}\cite{Marshall}
Let $ \Phi $ be a unital positive linear map. Then for $ A \in \mathcal{M}_n^{++} $\ satisfying $ 0 < h \leq A \leq k $ for some scalars $ h < k, $\\
\begin{equation}\label{e3}
\Phi(A^{-1} ) \leq \dfrac{(k+h)^{2}}{4kh}\Phi (A)^{-1}.
\end{equation}
\end{lemma}

 \begin{lemma}\label{A-G-ineq}\cite{Gumucs}
    Let \(A,B, X\in \mathcal{M}_n \) be such that \(A,B>O\) and let \(0<h<K\) be such that $ 0 < h \leq A \leq k $ for some scalars $ h < k, $ and let \(t=\min\lbrace \lambda, 1-\lambda\rbrace\). Then
    \[\mathcal{A}_{\bm{\lambda}}(A,B) \leq \Big(\dfrac{h\nabla_t k}{h\sharp_t k}\Big)\, A \sharp_{\lambda}\; B .\]
\end{lemma}

\begin{lemma}[Theorem 3.5, \cite{kubo}]\label{KA ineq}
    Let \(\sigma\) be a connection. Then, for every \(A,B,C,D\in M_n^{++}\), we have  \[ (A+C) \sigma (B+D) \geq (A\sigma B) + (C\sigma D).\]
\end{lemma}
The following lemma follows from Lemma \ref{KA ineq}.
\begin{lemma}
    Let \(A,B\in \mathcal{M}_n^{++}\). Then 
    \[J_{\mu^{-1} A} \sharp_{\bm{\lambda}}J_{\mu^{-1} B} \leq J_{\mu^{-1} (A\sharp_{\bm{\lambda}} B)},\] where \(J_{\mu^{-1} A}= (\mu^{-1} A+I)^{-1},\mu >0\).
\end{lemma} 

The following lemma states the relations between the real part of the inverse of $A$ and the inverse of $\Re A$ for a sectorial matrix.
\begin{lemma}\label{l1}(\cite{Lin1}, \cite{Lin2})
If $ A \in \Pi_{n,\alpha}$, then so is $A^{-1}$ and
\begin{equation*}
\Re(A^{-1})\leq\Re^{-1}(A)\leq \sec^2\alpha\ \Re(A^{-1}).
\end{equation*}
In particular, if \(A_i\in \Pi_{n,\alpha_i}\) and \(\alpha= \displaystyle\max_{1\leq i\leq m} \alpha_i \), then 
\[\Re \left(\sum_{i=1}^m\, A_i^{-1}\right)^{-1} \geq \cos^{2}(\alpha)\left(\sum_{i=1}^m\, \Re(A_i^{-1})\right)^{-1}  \geq \cos^{2}(\alpha)\left(\sum_{i=1}^m\, \Re^{-1}(A_i)\right)^{-1} .\]
\end{lemma}
\begin{remark}\label{sectorial index}
 It is evident that if, for \(1\leq i\leq m\), \(A_i\in \Pi_{n,\alpha_i}\), then \(A_i\in \Pi_{n,\alpha}\) where \(\alpha=\displaystyle \max_{1\leq i\leq m}\, \alpha_i\). So without loss of generality, we may assume that \(A_i\in \Pi_{n,\alpha}\) for all \(1\leq i\leq m\).   

Furthermore, we notice that when \(A_i\in\Pi_{n,\alpha}\), and $\mu \geq 0$, then $A_i+\mu I\in\Pi_{n,\gamma}$, where $\gamma=\tan^{-1}\frac{\tan\alpha}{1+\mu}.$ For the rest of this paper, $\gamma$ will denote the new angle obtained from $\alpha$ and $\mu$ by the formula $\gamma=\tan^{-1}\frac{\tan\alpha}{1+|\mu|}$. While $\gamma$ depends on $\alpha,\mu$, we will not explicitly add this to the notation as it will be clear from the context.
\end{remark}

\begin{lemma}\label{l-5}\cite{Rai}
Let $ A, B \in \Gamma_n $ and let \(\lambda\in (0,1)\). Then
$$ \Re A \sharp_{\lambda} \Re B \leq \Re (A \sharp B).  $$ 
\end{lemma}
\begin{lemma}\label{l-6}\cite{Tan}
Let \(A, B \in \Pi_{n,\alpha} \). Then
$$  \sec^{2} \alpha\, (\Re A \sharp \Re B ) \geq \Re (A \sharp B).  $$ 
\end{lemma}
\begin{lemma}\cite[lemma 2.4]{Law}\label{Riccati's Lemma}[Riccati's Lemma]
    Let \(A,B \in\mathcal{M}_n^{++}\). Then \(A\sharp B\) is the unique positive definite solution of the equation  \(XA^{-1}X=B\).
\end{lemma}
The accretive version of Lemma \ref{Riccati's Lemma} was established by Drury:
\begin{lemma}\label{Drury's Lemma}\cite[Proposition 3.5]{Drury}[Drury's Lemma]
    Let \(A,B,H\in \Gamma_n\). Then \(A\sharp B\) is the unique accretive solution of the equation  \(HA^{-1}H=B\). \\
\end{lemma}
In \cite{Bau},  the authors defined the resolvent average for positive definite matrices and proved that it enjoys some
interesting properties. 
\begin{definition}
    Let \(A_i\in \mathcal{M}_n^{++}, i=1,2,\cdots,m\). Then
the resolvent average of $ \textbf{A} $ with weight  $ \bm{\lambda} $ is defined as follows
\begin{equation}\label{e1}
\mathcal{R} _{\mu}(\textbf{A}, \bm{\lambda})= \left(  \sum _{i=1}^{m}\lambda_{i} (A_{i}+ \mu I)^{-1} \right)^{-1}-\mu I, \quad \mu \geq 0,
\end{equation}
and for \(\mu = \infty, \mathcal{R} _{\infty}(\textbf{A}, \bm{\lambda})= \mathcal{A}_{\bm{\lambda}}(\textbf{A})\). 
\end{definition}
Note that for \(\mu=0,  \mathcal{R} _{0}(\textbf{A}, \bm{\lambda})=\mathcal{H}_{\bm{\lambda}}(\textbf{A})\).\\

The following properties of the resolvent average can be found in \cite{Bau} and \cite{ku2} . \\
\begin{proposition}\label{properties of res average}
 Let \(A_i,B_i\geq O, i=1,2,\cdots,m\) and let \(\mu\geq 0\). Then
    \begin{enumerate}
    \item[1.] (\textbf{Idempotency}) \(\mathcal{R} _{\mu}((A,A,\cdots, A),\bm{\lambda})= A\).
    \item[2.] (\textbf{Self-duality}) \(\mathcal{R} _{\mu}(\textbf{A},\bm{\lambda})^{-1}=\mathcal{R} _{\mu^{-1}}(\textbf{A}^{-1},\bm{\lambda})\).
    \item[3.]  (\textbf{Unitary invariance}) \(U^{*} \mathcal{R} _{\mu}(\textbf{A},\bm{\lambda}) U= \mathcal{R} _{\mu}( U^{*}\textbf{A}U,\bm{\lambda}) \) for any unitary matrix \(U\in \mathcal{M}_n\), where \( U^{*}\textbf{A} U=(U^{*}A_{1}U, \ldots,U^{*}A_{n}U )\).
    \item[4.] (\textbf{Permutation invariance})
 \( \mathcal{R} _{\mu}(\textbf{A}_{\zeta} ,\bm{\lambda}_{\zeta})=\mathcal{R} _{\mu}(\textbf{A},\bm{\lambda})\), where \( \zeta\) is any permutation of \(\lbrace1,2,\cdots,m\rbrace, \textbf{A}_{\zeta}= (A_{\zeta(1)},A_{\zeta(2)},\cdots,A_{\zeta(m)})\) and \(\bm{\lambda}_{\zeta}= (\lambda_{\zeta(1)},\lambda_{\zeta(2)},\cdots,\lambda_{\zeta(m)})\).
    \item[5.]  (\textbf{Homogeneity}) \(\mathcal{R} _{\mu}( \alpha \textbf{A},\bm{\lambda})=\alpha \mathcal{R} _{\frac{\mu}{\alpha}}(\textbf{A},\bm{\lambda}), \alpha >0 \).
    \item[6.] \(\textbf{Monotonicity}\) \(\mathcal{R} _{\mu}(\textbf{B},\bm{\lambda} ) \geq \mathcal{R} _{\mu}(\textbf{A},\bm{\lambda}), \text{ if } B_i\geq A_i \text{ for } 1\leq i\leq m\).
    \item[7.] (\textbf{Monotonicity for Parameters})  \( \mathcal{R} _{\mu}(\textbf{A},\bm{\lambda}) \leq \mathcal{R} _{\nu}(\textbf{A},\bm{\lambda}), \text{ for } \mu\leq \nu\).
    \item[8.]  If \(\,\bm{\lambda}=(\frac{1}{2m}, \frac{1}{2m}, \ldots, \frac{1}{2m})\), then \(\mathcal{R} _{1}(A_{1}, A_{1}^{-1}, \ldots,  A_{m}, A_{m}^{-1},\bm{\lambda})=I \).
    \item[9.] (\textbf{Recursion}) \(\mathcal{R} _{\mu}(\textbf{A},\bm{\lambda})= \mathcal{R} _{\mu}\left(  \mathcal{R} _{\mu}\left( A_{1}, \ldots, A_{m-1},(\frac{\lambda_{1}}{1-\lambda}_{m}, \ldots, \frac{\lambda_{m-1}}{1-\lambda_{m}} )\right) ,A_{m},(1-\lambda_m,\lambda_m)\right)\).
\item[10.] (\textbf{Joint concavity}) For \(0\leq t\leq 1\), \[t\,\mathcal{R} _{\mu}(\textbf{A},\bm{\lambda} )+ (1-t)\, \mathcal{R} _{\mu}(\textbf{B},\bm{\lambda}) \leq \mathcal{R} _{\mu}((t\,\textbf{A}+(1-t)\,\textbf{B}),\bm{\lambda}).\]
\end{enumerate}
\end{proposition}
	
It follows from The monotonicity of the resolvent average that \( \mathcal{R} _{\mu}(\textbf{A},\bm{\lambda}) \geq O\). Furthermore, if \( A_{i} \in  \mathcal{M}_{n}^{++}\) for some \(i\), then $ \mathcal{R} _{\mu}(\textbf{A},\bm{\lambda}) \in \mathcal{M}_{n}^{++}. $\\

Moreover, the following Harmonic-Resolvent-Arithmetic mean inequality was proved in \cite[Theorem 4.2]{Bau} 
\begin{equation}\label{e0001}
\mathcal{H}_{\bm{\lambda}} (\textbf{A} ) \leq \mathcal{R} _{\mu}(\textbf{A},\bm{\lambda})  \leq  \mathcal{ A}_{\bm{\lambda}}(\textbf{A}).
\end{equation}

For other studies on the resolvent average, we refer to  \cite{Ba,  ku1, ku2, W1,W2}.

\begin{definition}
   Let \(A_i\in \mathcal{M}_n^{++}, i=1,2,\cdots,m\). Define for  \(\mu\geq 0\text{ and } \mu<0\), respectively
    \[\mathcal{L} _{\mu}( \textbf{A},\bm{\lambda})=  \left( \sum_{i=1}^{m} \bm{\lambda}_{i}(A_{i}+\mu I)\right) \sharp \left( \sum_{i=1}^{m} \bm{\lambda}_{i}(A_{i}+\mu I)^{-1}\right)^{-1}-\mu I.\]
   
    \[\mathcal{L} _{\mu}( \textbf{A},\bm{\lambda})= \mathcal{L}^{-1} _{-\mu}( \textbf{A}^{-1},\bm{\lambda}).\]
\end{definition}
It follows from the definition and the properties of the geometric mean of positive matrices that \(\mathcal{L} _{\mu}( \textbf{A},\bm{\lambda})\) is positive definite.\\

The following properties of the parametrized weighted \(\mathcal{A}\sharp \mathcal{H}\) can be found in \cite{kim}.
\begin{theorem}\label{prop of A-H mean}
  Let \(A_i\in \mathcal{M}_n^{++}, i=1,2,\cdots,m\) and let \(\nu,\mu\in [-\infty,\infty]\). Then 
\begin{enumerate}
\item[1.] (\textbf{Idempotency}) \(\mathcal{L} _{\mu}((A, \ldots,  A),\bm{\lambda})= A\).
    \item[2.] (\textbf{Self-duality}) \(\mathcal{L}^{-1} _{\mu}(\textbf{A}^{-1},\bm{\lambda})=\mathcal{L}_{-\mu}(\textbf{A},\bm{\lambda})\).
    \item[3.]  (\textbf{Unitary invariance}) \(U^{*} \mathcal{L} _{\mu}(\textbf{A},\bm{\lambda}) U= \mathcal{L} _{\mu}( U^{*}\textbf{A}U,\bm{\lambda}) \) for any unitary matrix \(U\in \mathcal{M}_n\), where \( U^{*}\textbf{A} U=(U^{*}A_{1}U, \ldots,U^{*}A_{n}U )\).
    \item[4.] (\textbf{Permutation invariance})
 \( \mathcal{L} _{\mu}(\textbf{A}_{\zeta}, \bm{\lambda}_{\zeta})=\mathcal{L} _{\mu}(\textbf{A},\bm{\lambda})\), where \( \zeta\) is any permutation of \(\lbrace1,2,\cdots,m\rbrace, \textbf{A}_{\zeta}= (A_{\zeta(1)},A_{\zeta(2)},\cdots,A_{\zeta(m)})\) and \(\bm{\lambda}_{\zeta}= (\lambda_{\zeta(1)},\lambda_{\zeta(2)},\cdots,\lambda_{\zeta(m)})\).
    \item[5.]  (\textbf{Homogeneity}) \(\mathcal{L} _{\mu}( \alpha \textbf{A},\bm{\lambda})=\alpha \mathcal{L} _{\frac{\mu}{\alpha}}(\textbf{A},\bm{\lambda}), \alpha >0 \).
    \item[6.] (\textbf{Monotonicity}) \(\mathcal{L} _{\mu}(\textbf{A},\bm{\lambda} ) \geq \mathcal{L} _{\mu}(\textbf{B},\bm{\lambda}), \text{ if } A_i\geq B_i \text{ for } 1\leq i\leq m\).
    \item[7.]\label{Resolvent-A-H mean ineq} \(\mathcal{R}_{\mu}(\textbf{A},\bm{\lambda})\leq \mathcal{L}_{\mu}(\textbf{A},\bm{\lambda}),\; \mu\geq 0\).
    \item[8.]  If \(\,\bm{\lambda}=(\frac{1}{2m}, \frac{1}{2m}, \ldots, \frac{1}{2m})\), then \(\mathcal{L} _{1}(A_{1}, A_{1}^{-1}, \ldots,  A_{m}, A_{m}^{-1},\bm{\lambda})=I \).
    \item[9.](\textbf{Monotonicity for Parameters})  \( \mathcal{L} _{\mu}(\textbf{A},\bm{\lambda}) \leq \mathcal{L} _{\nu}(\textbf{A},\bm{\lambda}), \text{ for } \mu\leq \nu\).
    \item[10.] \(\mathcal{L} _{\mu} (A,B,\frac{1}{2},\frac{1}{2}) = (A+\mu I)\sharp (B+\mu I)- \mu I, \;\mu\geq 0\).
    \item[11.] (\textbf{Joint concavity}) For \(0\leq t\leq 1\), \[t\,\mathcal{L} _{\mu}(\textbf{A},\bm{\lambda} )+ (1-t)\, \mathcal{L} _{\mu}(\textbf{B},\bm{\lambda}) \leq \mathcal{L} _{\mu}((t\,\textbf{A}+(1-t)\,\textbf{B}),\bm{\lambda}).\]
    \item[12.] \(\displaystyle \lim_{\mu \rightarrow \infty}\,\mathcal{L}_{\mu}(\textbf{A},\bm{\lambda})=\mathcal{A}_{\bm{\lambda}}(\textbf{A}) \text{ and }\displaystyle \lim_{\mu \rightarrow -\infty}\,\mathcal{L}_{\mu}(\textbf{A},\bm{\lambda})=\mathcal{H}_{\bm{\lambda}}(\textbf{A})\). 
\end{enumerate} 
\end{theorem}
In fact, we have 
\begin{equation}
\mathcal{H}_{\bm{\lambda}}(\textbf{A}) \leq \mathcal{L}_{\mu}(\textbf{A},\bm{\lambda})\leq \mathcal{A}_{\bm{\lambda}}(\textbf{A}).
\end{equation}
Together with [Theorem \ref{prop of A-H mean}-part 7] and equation \eqref{e0001}, we get 
\begin{equation}
    \mathcal{H}_{\bm{\lambda}}(\textbf{A}) \leq \mathcal{R}_{\mu}(\textbf{A},\bm{\lambda}
    )\leq  \mathcal{L}_{\mu}(\textbf{A},\bm{\lambda})\leq \mathcal{A}_{\bm{\lambda}}(\textbf{A}).
\end{equation}
An interesting property of the parametrized \(\mathcal{A}\sharp \mathcal{H}\)-mean follows from Riccati's Lemma.

\begin{theorem}\label{A-H mean and Riccati's Lemma}
    Let \(A_i\in \mathcal{M}_n^{++}, i=1,2,\cdots,m\) and let \(\mu \geq 0\). Then \(\mathcal{L}_{\mu}(\textbf{A},\bm{\lambda})+\mu I\) is the unique positive definite solution of the equation
    \[\sum_{i=1}^m\, \lambda_i \Big(X (A_i+\mu I)^{-1} X\Big)= \sum_{i=1}^m\, \lambda_i(A_i+\mu I).\]
\end{theorem}
The accretive version of Theorem \ref{A-H mean and Riccati's Lemma} is given in Theorem \ref{A-H mean and Drury's Lemma}.
\section{Results for Positive Definite Matrices}
In this section, we introduce new properties of both the resolvent average and the parameterized weighted \(\mathcal{A} \sharp \mathcal{H}\) mean for positive definite matrices.\\
To prove our results, the following lemma is needed.
\begin{lemma}
    Let \( A_{i}\in \mathcal{M}_n^{++}, i=1, \cdots, m \) and let  \(\mu \geq 0 \). Then for any positive unital linear map
 \(\Phi\),
\begin{equation}\label{eq-1}
\Phi\left(    \left(  \sum _{i=1}^{m}\lambda_{i}(A_{i}+\mu I)^{-1}   \right)^{-1}\right)  \leq  \left(  \sum _{i=1}^{m} \lambda_{i} (\Phi(A_{i})+ \mu I)^{-1} \right)^{-1}. 
\end{equation}
\end{lemma}

\begin{proof}
Letting \(\sigma\) be the Harmonic mean in Lemma \ref{l-h} and by applying the Principle of Mathematical Induction, we get
  $$ \Phi\left(    \left(  \sum _{i=1}^{m}A_{i}^{-1}   \right)^{-1}\right)  \leq  \left(  \sum _{i=1}^{m} \Phi(A_{i})^{-1} \right)^{-1}. $$
  
Now, replace $ A_{i}^{-1} $ with $ \lambda_{i}(A_{i}+\mu I)^{-1},\, i=1,\ldots, m, $ to get the desired result. 
\end{proof}
Using inequality (\ref{eq-1}), the following interesting inequality is obtained.
\begin{corollary}\label{phi-res-ineq 1}
  Let \( A_{i}\in \mathcal{M}_n^{++}, i=1, \cdots, m \) and let \(\mu >0 \). Then for any positive unital linear map
 \(\Phi\),
\begin{equation}\label{eq-2} 
  \Phi \left( \mathcal{R} _{\mu}( \textbf{A},\bm{\lambda}) \right) \leq \mathcal{R} _{\mu}( \Phi (\textbf{A},\bm{\lambda})).
\end{equation} 
 \end{corollary}
  In the following we obtain a reverse version of (\ref{eq-2}).
\begin{theorem}\label{phi-res-ineq 2}
Let \(A_{i}\in \mathcal{M}_n^{++}, i=1, \cdots, m \) such that \( 0< h \leq A_{i}\leq k,\) and let \(\mu >0 \). Then for any positive unital linear map
\(\Phi \),
\begin{equation*}
\frac{1}{\beta}\Phi \left( \mathcal{R} _{\beta \mu}( \textbf{A},\bm{\lambda}) \right) \geq \mathcal{R} _{\mu}( \Phi (\textbf{A}),\bm{\lambda}),
\end{equation*}
where, $\beta=\dfrac{4hk}{(k+h)^{2}}.$
\end{theorem}
\begin{proof}

\begin{align*}
\Phi \left( \mathcal{R} _{\mu^{-1}}( \textbf{A}^{-1},\bm{\lambda}) \right) &=\Phi \left( \mathcal{R} _{\mu}( \textbf{A})^{-1},\bm{\lambda} \right) \; (\text{by self-duality of } \mathcal{R}_{\mu}(.,.))\\
&\geq \Phi\left( \mathcal{R} _{\mu}( \textbf{A},\bm{\lambda}) \right)^{-1} \; (\text{by Lemma } \ref{l3})\\
&\geq \mathcal{R} _{\mu} \left( \Phi(\textbf{A}),\bm{\lambda} \right)^{-1}\; (\text{by } \eqref{eq-2})\\
&=\mathcal{R} _{\mu^{-1}} \left( \Phi(\textbf{A})^{-1},\bm{\lambda} \right)\\
&\geq \mathcal{R} _{\mu^{-1}}\left( \beta \Phi(\textbf{A}^{-1}) ,\bm{\lambda}\right)\; (\text{by Lemma } \ref{l4} \text{ and monotonicity of } \mathcal{R}_{\mu}(.,.))\\
&=\beta \mathcal{R} _{\frac{\mu^{-1}}{\alpha}} \left( \Phi(\textbf{A}^{-1}),\bm{\lambda} \right).
\end{align*}
Consequently, $\mathcal{R} _{\frac{\mu^{-1}}{\beta}} \left( \Phi(\textbf{A}^{-1}), \bm{\lambda} \right) \leq \frac{1}{\beta} \Phi \left( \mathcal{R} _{\mu^{-1}}( \textbf{A}^{-1}, \bm{\lambda}) \right).$
Now, replacing $A^{-1}, \mu^{-1}$ respectively, with $A, \beta \mu.$ So, the proof is complete.

 \end{proof}

\begin{theorem}\label{resolvent-geometric-mean}
    Let \(A,B\in \mathcal{M}_n^{++} \) and let \(0<h<k\) be such that \(0<h\leq A,B\leq k\). Then, for $0\leq \lambda\leq 1,$
    \[\mathcal{R}_{\mu}(A,B,1-\lambda,\lambda)+\mu I \geq C \Big(A \sharp_{\bm{\lambda}} B +\mu I\Big),\]
    where \(C=\Big(\dfrac{k^*\nabla_t h^*}{k^*\sharp_t h^*}\Big)^{-1}, k^*= (k+\mu )^{-1} \text{ and } h^*=(h+\mu )^{-1} \).
 \end{theorem}
 \begin{proof}
 Let \(A,B\) be such that  \(0<h \leq A,B\leq k \). Then \[k^* \leq (A+\mu I)^{-1}, (B+\mu I)^{-1}\leq h^*,\]
 where \(k^*= (k+\mu )^{-1} \text{ and } h^*=(h+\mu )^{-1} \). 
 Hence, with \(t=\min\lbrace \lambda, 1-\lambda \rbrace\), we have by Lemma \ref{A-G-ineq}
 \[ \mathcal{A}_{\bm{\lambda}}((A+\mu I)^{-1}, (B+\mu I)^{-1}) \leq \dfrac{k^*\nabla_t h^*}{k^*\sharp_t h^*}\, \Big((A+\mu I)^{-1}\sharp_{\bm{\lambda}} (B+\mu I)^{-1}\Big).\]
 That is, 
 \begin{align} \label{reverse of res A-H mean ineq}
 \begin{split}
 \mathcal{R}_{\mu}(A,B,1-\lambda,\lambda)+\mu I &= \mathcal{A}_{\bm{\lambda}}^{-1}((A+\mu I)^{-1}, (B+\mu I)^{-1}) \\
 &\geq \Big(\dfrac{k^*\nabla_t h^*}{k^*\sharp_t h^*}\Big)^{-1}\, \Big((A+\mu I)^{-1}\sharp_{\lambda} (B+\mu I)^{-1}\Big)^{-1} \\
 &= C  \Big((A+\mu I) \sharp_{\lambda} (B+\mu I)\Big)\\
&\geq C \Big( (A \sharp_{\lambda}\, B) + (\mu I \sharp_{\lambda} \mu I)\Big)\; (\text{ by Lemma \ref{KA ineq} })\\
&= C \Big( (A \sharp_{\lambda}\, B) + \mu I\Big),\\
 \end{split}
 \end{align}
 where \(C=\Big(\dfrac{k^*\nabla_t h^*}{k^*\sharp_t h^*}\Big)^{-1} \).
 \end{proof}
\begin{remark}
    Notice that if \(\bm{\lambda}=(\frac{1}{2},\frac{1}{2})\), then by part \(7\) and part \(10\) of Theorem \ref{prop of A-H mean}, we have
    \begin{equation}\label{R(A,B)-L(A,B) ineq}
        \mathcal{R}_{\mu}\left(A,B,\frac{1}{2},\frac{1}{2}\right)+\mu I \leq (A+\mu I) \sharp_{\lambda} (B+\mu I).   
        \end{equation}
    Under the assumptions of Theorem \ref{resolvent-geometric-mean},  inequality \eqref{reverse of res A-H mean ineq} provides a reverse version of \eqref{R(A,B)-L(A,B) ineq}.
\end{remark}

\begin{remark}
It follows from Theorem \ref{resolvent-geometric-mean} that \[R_{\mu}(A,B,1-\lambda,\lambda)+\mu I \geq C \Big(\dfrac{h\nabla_t k}{h\sharp_t k} \Big)^{-1}\mathcal{A}_{\bm{\lambda}}(A,B)+C \mu I.\]
That is,
\[R_{\mu}(A,B,1-\lambda,\lambda) \geq C \Big(\dfrac{h\nabla_t k}{h\sharp_t k} \Big)^{-1}\mathcal{A}_{\bm{\lambda}}(A,B) - \mu \Big(1-C\Big) I \]
\end{remark}

 \begin{theorem}\label{t-15}
 Let \(A_{i}\in \mathcal{M}_n^{++}, i=1, \cdots, m \) be such that \( 0< h \leq A_{i}\leq k, \, \mu \in [-\infty,\infty] \). Then for any positive unital linear map
 $ \Phi, $
 \begin{equation}\label{eq-t-15}
     \Phi \left( \mathcal{L} _{\mu}( \textbf{A}, \bm{\lambda}) \right) \leq \mathcal{L} _{\mu}( \Phi (\textbf{ A }), \bm{\lambda}).
 \end{equation}
 \end{theorem}
 
 \begin{proof}
 Assume first that \(\mu \geq 0\). Then, by Ando's inequality and \eqref{eq-1}, we obtain,
 \begin{align*}
 & \Phi \left(   \left( \sum_{i=1}^{m} \lambda_{i}(A_{i}+\mu I)\right) \sharp \left( \sum_{i=1}^{m} \lambda_{i}(A_{i}+\mu I)^{-1}\right)^{-1}-\mu I \right)\\ 
 & \leq \Phi \left( \sum_{i=1}^{m} \lambda_{i}(A_{i}+\mu I)\right) \sharp \Phi 
 \left( \left( \sum_{i=1}^{m}  \lambda_{i}(A_{i}+\mu I)^{-1}\right)^{-1} \right)-\mu I  \\
& \leq  \left( \sum_{i=1}^{m} \lambda_{i}(\Phi(A_{i})+\mu I)\right) \sharp  \left( \sum_{i=1}^{m}  \lambda_{i}(\Phi(A_{i})+\mu I)^{-1}\right)^{-1} -\mu I \\
&=  \mathcal{L} _{\mu}( \Phi (\textbf{ A }), \bm{\lambda}).
 \end{align*}
Now, if \(\mu<0\), then 
\begin{align*}
    \Phi \left( \mathcal{L} _{\mu}( \textbf{A}, \bm{\lambda}) \right) & =\Phi \left( \mathcal{L}^{-1} _{-\mu}( \textbf{A}^{-1}, \bm{\lambda})
    \right)\\
     & \geq \Phi^{-1} \left( \mathcal{L} _{-\mu}( \textbf{A}^{-1}, \bm{\lambda}) \right)\, (\text{by Lemma \ref{l3}})\\
    &\geq \mathcal{L}^{-1} _{-\mu}\left( \Phi (\textbf{A}^{-1}), \bm{\lambda} \right)\\
    &\geq \mathcal{L}^{-1} _{-\mu}\left( \Phi^{-1} (\textbf{A}), \bm{\lambda} \right)\, (\text{by the Monotonicity of \(\mathcal{L}_{\mu}(.,.)\)})\\
    &= \mathcal{L} _{\mu}\left( \Phi(\textbf{A}), \bm{\lambda} \right),
\end{align*}
which completes the proof.
\end{proof}
In the following, we obtain a reverse version of (\ref{eq-t-15}).
\begin{theorem}
Let \( A_{i}\in \mathcal{M}_n^{++}, i=1, \cdots, m \) be such that \( 0< h \leq A_{i}\leq k\) and let \(\mu \in [-\infty, +\infty ] \). Then for any positive unital linear map
\(\Phi\)
\begin{equation*}
\frac{1}{\beta}\Phi \left( \mathcal{L} _{\beta \mu}( \textbf{A}, \bm{\lambda}) \right) \geq \mathcal{L} _{\mu}( \Phi (\textbf{ A }), \bm{\lambda}),
\end{equation*}
where, $\beta=\dfrac{4hk}{(k+h)^{2}}.$
\end{theorem}
\begin{proof}
Let \(\mu\in [-\infty,\infty]\). Then, we have
\begin{align*}
\Phi \left( \mathcal{L} _{-\mu}( \textbf{A}^{-1}, \bm{\lambda}) \right) &=\Phi \left( \mathcal{L}^{-1} _{\mu}( \textbf{A}, \bm{\lambda}) \right)\, (\text{by self-duality of }\mathcal{L}_{\mu}(.,.))\\
&\geq \Phi^{-1}\left( \mathcal{L} _{\mu}( \textbf{A}, \bm{\lambda}) \right)\, (\text{by Lemma \ref{l3}}) \\
&\geq \mathcal{L}^{-1} _{\mu} \left( \Phi(\textbf{A}), \bm{\lambda} \right)\, (\text{by \eqref{eq-t-15}})\\
&=\mathcal{L} _{-\mu} \left( \Phi(\textbf{A})^{-1}, \bm{\lambda} \right)\\
&\geq \mathcal{L} _{-\mu }\left( \beta \Phi(\textbf{A}^{-1}), \bm{\lambda} \right)\, (\text{by Lemma \ref{l4} and the Monotonicity of } \mathcal{L}_{\mu}(.,.))\\
&=\beta \mathcal{L} _{\frac{-\mu}{\alpha}} \left( \Phi(\textbf{A}^{-1}), \bm{\lambda} \right)\, (\text{by the homogeneity of } \mathcal{L}_{\mu}(.,.)).
\end{align*}
Consequently, $\mathcal{L} _{\frac{-\mu}{\beta}} \left( \Phi(\textbf{A}^{-1}), \bm{\lambda} \right) \leq \frac{1}{\beta} \Phi \left( \mathcal{L} _{-\mu}( \textbf{A}^{-1}, \bm{\lambda}) \right).$
Now, replacing $A^{-1}, -\mu$ respectively, with $A,\, \beta \mu$ completes the proof. 
\end{proof}

\section{Resolvent Average of Accretive Matrices}
In this section we generalize the definition of the resolvent average to the accretive case and prove some properties.\\

\begin{definition}
  Let \(A_i \in \Pi_{n,\alpha}\,, i=1,2,\cdots,m\) and let \(\mu \geq 0\). Then the resolvent average of $ \textbf{A} $ with weight  $ \bm{\lambda} $ is given by
  \begin{equation}\label{e1}
\mathcal{R} _{\mu}(\textbf{A}, \bm{\lambda} )= \left(  \sum _{i=1}^{m} \lambda_{i} (A_{i}+ \mu I)^{-1} \right)^{-1}-\mu I,
\end{equation}
\end{definition}
It follows from the definition that the resolvent of \(\mu \mathcal{R}_{\mu}(\textbf{A},\bm{\lambda})\) is the weighted arithmetic mean of the resolvent of \(\mu\textbf{A}\). That is, \[\mathcal{J}_{\mu \mathcal{R}_{\mu}(\textbf{A},\bm{\lambda})}= \sum_{i=1}^{m}\, \lambda_i \mathcal{J}_{\mu A_i},\]
where \(\mathcal{J}_{\mu A_i}= (\mu A_i+I)^{-1}\).\\
\begin{remark}
Clearly, \(\mathcal{R} _{\mu}(\textbf{A}, \bm{\lambda} )+\mu I\) is accretive, whenever \(A_i, i=1,2,\cdots,m\) are accretive. In fact, if \(A_i\in \Pi_{n,\alpha}\) for each \(i\), then by Lemma \ref{l1} and Remark \ref{sectorial index}, \(\mathcal{R} _{\mu}(\textbf{A}, \bm{\lambda} )+\mu I \in \Pi_{n, \gamma}\), where \(\gamma = \tan^{-1}\frac{\tan\alpha}{1+|\mu|}\).    
\end{remark}

The following are basic properties of the resolvent average that follow from the definition. 
\begin{proposition}
 Let \(A_i\in \Gamma_n, i=1,2,\cdots,m\) and let \(\mu\geq 0\). Then
    \begin{enumerate}
    \item[1.] \(\mathcal{R} _{\mu}((A, A,\ldots, A),\bm{\lambda})= A\).
    \item[2.]  \(\mathcal{R} _{\mu}(\textbf{A},\bm{\lambda})^{-1}=\mathcal{R} _{\mu^{-1}}(\textbf{A}^{-1},\bm{\lambda})\).
    \item[3.]  \(U^{*} \mathcal{R} _{\mu}(\textbf{A},\bm{\lambda}) U= \mathcal{R} _{\mu}( U^{*}\textbf{A}U,\bm{\lambda}) \) for any unitary matrix \(U\in \mathcal{M}_n\), where \( U^{*}\textbf{A} U=(U^{*}A_{1}U, \ldots,U^{*}A_{n}U )\).
    \item[4.]  \( \mathcal{R} _{\mu}(\textbf{A}_{\zeta} ,\bm{\lambda}_{\zeta})=\mathcal{R} _{\mu}(\textbf{A},\bm{\lambda})\), where \( \zeta\) is any permutation of \(\lbrace1,2,\cdots,m\rbrace, \textbf{A}_{\zeta}= (A_{\zeta(1)},A_{\zeta(2)},\cdots,A_{\zeta(m)})\) and \(\bm{\lambda}_{\zeta}= (\lambda_{\zeta(1)},\lambda_{\zeta(2)},\cdots,\lambda_{\zeta(m)})\).
    \item[5.]  \(\mathcal{R} _{\mu}( \alpha \textbf{A},\bm{\lambda})=\alpha \mathcal{R} _{\frac{\mu}{\alpha}}(\textbf{A},\bm{\lambda}), \alpha >0 \).
    \item[6.] If \(\,\bm{\lambda}=(\frac{1}{2m}, \frac{1}{2m}, \ldots, \frac{1}{2m})\), then \(\mathcal{R} _{1}(A_{1}, A_{1}^{-1}, \ldots,  A_{m}, A_{m}^{-1},\bm{\lambda})=I \).
    \item[7.] \(\mathcal{R} _{\mu}(\textbf{A},\bm{\lambda})= \mathcal{R} _{\mu}\left(  \mathcal{R} _{\mu}\left( A_{1}, \ldots, A_{m-1},(\frac{\lambda_{1}}{1-\lambda}_{m}, \ldots, \frac{\lambda_{m-1}}{1-\lambda_{m}} )\right) ,A_{m},(1-\lambda_m,\lambda_m)\right)\).
\end{enumerate}
\end{proposition}
By inequality (\ref{eq-2}), we conclude that if $ A_{i}, i=1,2, \cdots, m$ are accretive and $ \Phi $ is unital positive linear map,  then
\begin{equation}\label{e4}
 \Phi (\mathcal{R} _{\mu}(\Re \textbf{A}, \bm{\lambda} ))  \leq  \mathcal{R} _{\mu} (\Phi (\Re \textbf{A}), \bm{\lambda}).
\end{equation}
The following theorem is a reverse version of (\ref{e4}).
\begin{theorem}
Let $ A_{i}\in \Gamma_n, i=1, \cdots, m $  be  such that \\ $ h \leq \Re A_{i}+\mu I \leq k, \, \mu \geq 0 $ and $ \Phi $ be a  unital positive linear map. Then
\[  \dfrac{4kh}{(k+h)^{2}}\mathcal{R} _{\mu} (\Phi (\Re \textbf{A}), \bm{\lambda}) \leq \Phi (\mathcal{R} _{\mu}(\Re \textbf{A}, \bm{\lambda} ))+\left( \dfrac{k-h}{k+h} \right)^{2}\mu I . \]
In particular, if $ h\leq \Re A_{i} \leq k, $
$$  \mathcal{H}_{\bm{\lambda}} \left(\Phi(\Re \textbf{A})\right)\leq 
 \dfrac{(k+h)^{2}}{4kh}\Phi \left( \mathcal{H}_{\bm{\lambda}} (\Re \textbf{A})\right). $$
\end{theorem}
\begin{proof}
By Lemma \ref{l3} and Lemma \ref{l4}, we have
\begin{align*}
\Phi (\mathcal{R} _{\mu}(\Re \textbf{A}, \bm{\lambda} )) &= \Phi \left( \left(\sum_{i=1}^{m}\lambda_{i}( \Re A_{i}+\mu I)^{-1}\right)^{-1} \right)-\mu I\\
&\geq \Phi^{-1} \left( \sum_{i=1}^{m}\lambda_{i}( \Re A_{i}+\mu I)^{-1} \right)-\mu I\\
&=\left( \sum_{i=1}^{m}\, \lambda_{i}\Phi \left(( \Re A_{i}+\mu I)^{-1}\right)  \right)^{-1}-\mu I\\
&\geq \left( \sum_{i=1}^{m}\, \lambda_{i} \dfrac{(M+m)^{2}}{4Mm}\Phi \left(  \Re A_{i}+\mu I \right)^{-1}  \right)^{-1}-\mu I\\
&= \dfrac{4kh}{(k+h)^{2}}  \left( \sum_{i=1}^{m}\,\lambda_{i} (\Phi \left(  \Re A_{i})+\mu I \right)^{-1}  \right)^{-1}-\mu I\\
&= \dfrac{4kh}{(k+h)^{2}} \mathcal{R} _{\mu} (\Phi (\Re \textbf{A}), \bm{\lambda})-\left( \dfrac{k-h}{k+h} \right)^{2}\mu I.
\end{align*}
Then, the proof is complete.
\end{proof}

\begin{corollary}
Let $ A_{i}\in \Gamma_n, i=1, \cdots, m $ are such that $ h \leq \Re A_{i}+\mu I \leq k$, and let $\mu \geq 0 $. Then if $ X \in  \mathcal{M}_n$ is an isometry, i.e. $ X^{*}X=I $, then
\[  \dfrac{4kh}{(k+h)^{2}}\mathcal{R} _{\mu} (X^{*}\Re \textbf{A}X, \bm{\lambda}) \leq X^{*}\mathcal{R} _{\mu}(\Re \textbf{A}, \bm{\lambda} )X+\left( \dfrac{k-h}{k+h} \right)^{2}\mu I . \]
In particular, if $ h\leq \Re A_{i} \leq k , $
$$  \mathcal{H}_{\bm{\lambda}} \left( X^{*}\Re \textbf{A}X\right)\leq 
 \dfrac{(k+h)^{2}}{4kh}  X^{*}\mathcal{H} _{\bm{\lambda}}(\Re \textbf{A})X. $$
\end{corollary}

\begin{example}
If \(A_i\in \Gamma_n, i=1,2,\cdots,m\), then \(\Re  (\mathcal{R} _{\mu}(\textbf{A}, \bm{\lambda})\leq  (\mathcal{R} _{\mu}(\Re \textbf{A}, \bm{\lambda} )) \) does not hold in general. To see that, set $ A_{1}=I+iI, A_{2}=I, \bm{\lambda}_{1}=\bm{\lambda}_{2}=\frac{1}{2} $ and $ \mu=1. $ 
It is easy to see $ \Re \left(\mathcal{R} _{\mu}( \textbf{A}, \bm{\lambda} )\right)=\frac{36-\sqrt{17}}{\sqrt{17}}I $ but $   \mathcal{R} _{\mu}( \Re\textbf{A}, \bm{\lambda} )=I$. 
\end{example}


\begin{theorem}\label{t1}
Let \(A_i\in \Pi_{n,\alpha}, i=1,2,\cdots,m,\). Then
\[     \mathcal{R} _{\mu}(\Re \textbf{A}, \bm{\lambda} ) + \mu I   \leq  \sec^{2} \gamma \, \left(\Re \mathcal{R} _{\mu}(\textbf{A}, \bm{\lambda} )+\mu  I\right). \]
In particular, as $ \mu \rightarrow 0^+ $ 
\[ \mathcal{H} _{\bm{\lambda}}(\Re \textbf{A} ) \leq \sec^{2} \gamma\, \Re \Big(\mathcal{H}_{\bm{\lambda}}( \textbf{A} )\Big).\]
\end{theorem}

\begin{proof}
By Lemma \ref{l1},  we have
\begin{align*}
\mathcal{R} _{\mu}(\Re \textbf{A}, \bm{\lambda} ) &=  \left ( \sum_{i=1}^{m} \, \lambda_{i} (\Re (A_{i}) +\mu I)^{-1} \right)^{-1}-\mu I\\
&=    \left ( \sum_{i=1}^{m} \lambda_{i} \Re^{-1} (A_{i}+\mu I) \right)^{-1}-\mu I\\
& \leq  \left ( \sum_{i=1}^{m} \lambda_{i} \Re  \left( ( A_{i}+\mu I )^{-1}\right)  \right)^{-1}-\mu I\\
&= \Re^{-1} \left(  \sum_{i=1}^{m}\, \lambda_{i}  ( A_{i}+\mu I )^{-1}\right) -\mu I\\ 
&\leq \sec^{2}\gamma \, \Re \left( \left( \sum_{i=1}^{m}\,\lambda_{i} ( A_{i}+\mu I )^{-1} \right)^{-1}\right) -\mu I \\
&= \sec^{2}\gamma \, \Re \mathcal{R} _{\mu}(\textbf{A}, \bm{\lambda} )+\mu (\sec^{2}\gamma -1)I.
\end{align*}

Hence, the proof is complete.
\end{proof}

Theorem \ref{t2} provides a reverse version of Theorem \ref{t1}.

\begin{theorem}\label{t2}
Let \(A_i\in \Pi_{n,\alpha}, i=1,2,\cdots,m,\). Then

$$ \Re \mathcal{R} _{\mu}(\textbf{A}, \bm{\lambda} )+\mu I  \leq \sec^{2}\gamma \left( \mathcal{R} _{\mu}(\Re\textbf{A}, \bm{\lambda} )+ \mu I \right). $$
In particular, as $ \mu \rightarrow 0^{+} $ 
\[ \Re \mathcal{H}_{\bm{\lambda}}(\textbf{A}) \leq \sec^{2} \gamma\, \mathcal{H}_{\bm{\lambda}} (\Re \textbf{A} ).\]
\end{theorem}

\begin{proof}
By Lemma \ref{l1},  we have
\begin{align*}
\mathcal{R} _{\mu}(\Re \textbf{A}, \bm{\lambda} ) &=  \left ( \sum_{i=1}^{m} \lambda_{i} (\Re (A_{i}) +\mu I)^{-1} \right)^{-1}-\mu I\\
&=  \left ( \sum_{i=1}^{m} \lambda_{i} \Re^{-1} (A_{i}+\mu I) \right)^{-1}-\mu I\\
&\geq\left ( \sum_{i=1}^{m} \lambda_{i} \sec^{2}\gamma\, \Re  \left( ( A_{i}+\mu I )^{-1}\right)  \right)^{-1}-\mu I\\
&=\cos^{2}\gamma \,\Re^{-1} \left (  \sum_{i=1}^{m} \lambda_{i}    ( A_{i}+\mu I )^{-1} \right)-\mu I\\
&\geq \cos^{2}\gamma\, \Re \left ( \left ( \sum_{i=1}^{m} \lambda_{i}    ( A_{i}+\mu I )^{-1} \right)^{-1}\right)-\mu I\\
&=\cos^{2}\gamma\, \Re \mathcal{R} _{\mu}(\textbf{A}, \bm{\lambda} )-( \mu \sin^{2} \gamma )I.
\end{align*}
Hence, the proof is complete.
\end{proof}

\begin{corollary}
Let \(A_i\in \Pi_{n,\alpha}, i=1,2,\cdots,m,\). Then
\begin{equation}\label{c1}
 \cos^{2} \gamma \, \Re \mathcal{R} _{\mu}(\textbf{A}, \bm{\lambda} )- \mu \sin^{2}\gamma\, I  \leq \mathcal{R} _{\mu}(\Re \textbf{A}, \bm{\lambda} )    \leq  \sec^{2} \gamma \, \Re \mathcal{R} _{\mu}(\textbf{A}, \bm{\lambda} )+ \mu \tan^{2}\gamma\, I,
\end{equation}
and
\begin{equation}\label{c2}
\cos^{2} \gamma \, \mathcal{R} _{\mu}(\Re \textbf{A}, \bm{\lambda} ) -\mu \sin^{2}\gamma\, I  \leq  \Re \mathcal{R} _{\mu}(\textbf{A}, \bm{\lambda} )  \leq  \sec^{2} \gamma \, \mathcal{R} _{\mu} \Re \textbf{A}, \bm{\lambda} )+ \mu \tan^{2}\gamma \, I.
\end{equation}
\end{corollary}
A generalization of \eqref{e0001} form the setting of positive definite matrices to sectorial ones is given by the following corollary. 
\begin{corollary}
Let \(A_i\in \Pi_{n,\alpha}, i=1,2,\cdots,m,\). Then
\begin{align*}
\cos^{2}\gamma  \mathcal{H}( \Re \textbf{A}, \bm{\lambda} )-\mu \sin^{2}\gamma \,I & \leq \Re\mathcal{R} _{\mu}(\textbf{A}, \bm{\lambda} )\\
&\leq \sec^{2}\gamma \Re \mathcal{ A}(\textbf{A}, \bm{\lambda} )+\mu\tan^{2}\gamma \,I.\\ 
\end{align*}
\end{corollary}

\begin{remark}
It is clear that $ \Re \mathcal{A}_{\bm{\lambda}}(\textbf{A} )= \mathcal{A}_{\bm{\lambda}}(\Re \textbf{A} ), $ but 
by combining  \eqref{e0001} for $ \Re \textbf{A}$ and (\ref{c2}), we get other accretive versions of \eqref{e0001}.
\end{remark}
As an application of Theorem \ref{t1}, we present the following accretive version of Theorem \ref{resolvent-geometric-mean}.
\begin{theorem}
     Let \(A,B\in \Pi_{n,\alpha}\) be such that \(0< h\leq \Re(A), \Re(B)\leq k\) for some \(0<h<k\), and let \(t=\min \lbrace \lambda,1-\lambda\rbrace\). Then \[\Re(\mathcal{R}_{\mu}(A,B,\bm{\lambda}))+\mu I \geq C \cos^2 \gamma\, \Big(\cos^2 \alpha\, \Re(A\sharp_{\bm{\lambda}}\, B) + \mu I\Big),\]
where \(C=\Big(\dfrac{k^*\nabla_t h^*}{k^*\sharp_t h^*}\Big)^{-1} \).
 \end{theorem}
 \begin{proof}
     \begin{align*}
         \Re(\mathcal{R}_{\mu}(A,B,\bm{\lambda}))+\mu I  &= \Re\Big(\mathcal{R}_{\mu}(A,B,\bm{\lambda})+\mu I\Big)\\
        & \geq \cos^2 \gamma\,  \Big(\mathcal{R}_{\mu}(\Re(A),\Re(B),\bm{\lambda})+\mu I\Big)\; (\text{by Theorem \ref{t1}})\\
        & \geq C \cos^2 \gamma\, \Big( \Re(A)\sharp_{\bm{\lambda}} \Re(B) +\mu I\Big)\; (\text{by Theorem \ref{resolvent-geometric-mean}})\\
        &\geq C \cos^2 \gamma\, \Big( \cos^2 \alpha\, \Re(A\sharp_{\bm{\lambda}}\, B) +\mu I\Big) \, (\text{ by Lemma \ref{l-6}}).
\end{align*}
 \end{proof}
We conclude this section with the following interesting property of the resolvent average for two matrices.\\
If $ \textbf{A}= (A, B), \, \bm{\lambda} = (\bm{\lambda}, 1-\bm{\lambda}), $ then for $ \mu > 0 $, we have
\begin{equation}\label{e4-1}
\mathcal{R} _{\mu}(\textbf{A}, \bm{\lambda} )= \left(\lambda (A+\mu I)^{-1}+ (1-\lambda) (B+\mu I)^{-1} \right)^{-1}-\mu I.
\end{equation}
 \begin{theorem}
 For positive definite matrices $ A, B \in \mathcal{M}_n, \mu > 0$ and 
 $ 0 < \lambda <1,  $ the relation 
 
  $$ \mathcal{R}_{\mu}(\textbf{A},\bm{\lambda})= \left(\lambda (A+\mu I)^{-1}+ (1-\lambda) (B+\mu I)^{-1} \right)^{-1}-\mu I$$
 is a matrix mean.
 \end{theorem}
 \begin{proof}
     Let \(A,B\in \mathcal{M}_n^{++}, \,\mu >0 \text{ and } \lambda\in (0,1)\). For \(t\in (0,1)\), set 
     \[f(t)= \left(\lambda (1+\mu )^{-1}+ (1-\lambda) (t+\mu )^{-1} \right)^{-1}-\mu.\]

Since the function $t\mapsto t^{-1}$ is matrix monotone decreasing on $(0,\infty)$, it follows that $f$ is matrix monotone increasing with $f(1)=1$. Therefore, the relation
\[(A,B)\mapsto A^{\frac{1}{2}}f\left(A^{-\frac{1}{2}}BA^{-\frac{1}{2}}\right)A^{\frac{1}{2}}\]
is indeed a matrix mean. But, in this case, 
\[A^{\frac{1}{2}}f\left(A^{-\frac{1}{2}}BA^{-\frac{1}{2}}\right)A^{\frac{1}{2}}=\mathcal{R}_{\mu}(\textbf{A},\bm{\lambda}),\]
which completes the proof.
\end{proof}

\section{Weighted \(\mathcal{A} \sharp \mathcal{H}\)-Mean of Accretive Matrices}
In this section, we extend the definition of the parametrized weighted \(\mathcal{A} \sharp \mathcal{H}\) to the accretive case, and prove some properties.
\begin{definition}
 Let \( A_{i}\in \Gamma_n, i=1, \ldots, m \). Define for \(\mu\geq 0, \mu  <0 \), respectively   
\[\mathcal{L}_{\mu}(\textbf{A},\bm{\lambda})= \Big(\displaystyle \sum_{i=1}^{m}\, \lambda_i (A_i+\mu I)\Big) \sharp \Big(\displaystyle \sum_{i=1}^{m}\, \lambda_i (A_i+\mu I)^{-1}\Big)^{-1} -\mu I.\]
\[\mathcal{L}_{\mu}(\textbf{A},\bm{\lambda})= \mathcal{L}^{-1}_{-\mu}(\textbf{A}^{-1})\;  (\mu <0).\]
If \(\mu=0\), then \(\mathcal{L}(\textbf{A},\bm{\lambda}) =\mathcal{L}_0(\textbf{A},\bm{\lambda})= \mathcal{A}_{\bm{\lambda}}(\textbf{A}) \sharp \mathcal{H}_{\bm{\lambda}}(\textbf{A})\) is called the weighted \(\mathcal{A} \sharp \mathcal{H}\) mean and \(\mathcal{L}_{\mu}(\textbf{A},\bm{\lambda})\) is called the weighted \(\mathcal{A} \sharp \mathcal{H}\) mean of parameter \(\mu\).
\end{definition}
\begin{remark}
Clearly, \(\mathcal{L} _{\mu}(\textbf{A}, \bm{\lambda} )+\mu I\) is accretive, whenever \(A_i, i=1,2,\cdots,m\) are accretive. In fact, if \(A_i\in \Pi_{n,\alpha}\) for each \(i\), then by Lemma \ref{l1} and Remark \ref{sectorial index}, \(\mathcal{L} _{\mu}(\textbf{A}, \bm{\lambda} )+\mu I \in \Pi_{n, \gamma}\), where \(\gamma = \tan^{-1}\frac{\tan\alpha}{1+|\mu|}\).    
\end{remark}

Theorem \ref{G(A,B)=G(Arith,H)-Accretive} below is an analogue of of Lemma \ref{G(A,H)=G(A,B)} generalizing the result from positive matrices to accretive ones.
\begin{theorem}\label{G(A,B)=G(Arith,H)-Accretive}
    Let \(A,B\in \Gamma_n\). Then 
    \[(A+B)\sharp (A^{-1}+B^{-1})^{-1}= A\sharp B.\]
\end{theorem}
\begin{proof}
    Since \(A\sharp B= B \sharp A\), we have
    \begin{align*}
       ( A\sharp B) (A^{-1}+B^{-1}) ( A\sharp B) &=  ( A\sharp B) (A^{-1}( A\sharp B) + (B \sharp A)B^{-1} (B \sharp A)\\
       &= A^{1/2}\Big(A^{-1/2} B A^{-1/2}\Big)^{1/2}A^{1/2} A^{-1} A^{1/2}\Big(A^{-1/2} B A^{-1/2}\Big)^{1/2}A^{1/2} \\ &+ B^{1/2}\Big(B^{-1/2} A B^{-1/2}\Big)^{1/2}B^{1/2} B^{-1} B^{1/2}\Big(B^{-1/2} A B^{-1/2}\Big)^{1/2}B^{1/2}\\
       &= A^{1/2}\Big(A^{-1/2} B A^{-1/2}\Big)A^{1/2} + B^{1/2}\Big(B^{-1/2} A B^{-1/2}\Big)B^{1/2}\\
       & = B+A.
    \end{align*}
    Therefore, by Drury's Lemma and the commutativity of the geometric mean, \(A\sharp B=  (A+B)\sharp (A^{-1}+B^{-1})^{-1}\).
\end{proof}
\begin{corollary}
   Let \(A,B\in \Gamma_n\) and let \(\bm{\lambda}=(1/2,1/2)\). Then \[\mathcal{L}_\mu(A,B,\bm{\lambda})= \left((A+\mu I)\sharp(B+\mu I)\right)-\mu I.\]
\end{corollary}
\begin{proof}
\begin{align*}\mathcal{L}_\mu(A,B,\bm{\lambda}) & =  \left(\frac{1}{2}(A+\mu I) + \frac{1}{2}(B+\mu I) \right) \sharp \left(\frac{1}{2}(A+\mu I)^{-1} + \frac{1}{2}(B+\mu I)^{-1} \right)^{-1} -\mu I\\
&= \left((A+\mu I) + (B+\mu I) \right) \sharp \left((A+\mu I)^{-1} + (B+\mu I)^{-1} \right)^{-1} -\mu I\\
&= \left((A+\mu I)\sharp (B+\mu I)\right) -\mu I \; (\text{by Theorem \ref{G(A,B)=G(Arith,H)-Accretive}}).
\end{align*}
\end{proof}
The following are basic properties of the parametrized weighted \(\mathcal{A} \sharp \mathcal{H}\) mean that follows directly from the definition and some properties of the weighted geometric mean of accretive matrices.
\begin{proposition}
Let \(A_i\in \Gamma_n, i=1,2,\cdots,m\) and let \(\mu\in[-\infty,\infty] \). Then 
   \begin{enumerate}
       \item[1.] \(\mathcal{L}_{\mu}((A,A,\ldots, A),\bm{\lambda})= A\).
       \item[2.] \(\mathcal{L}_{\mu}(k \textbf{A},\bm{\lambda})= k \mathcal{L}_{\mu/k}(\textbf{A},\bm{\lambda}),\, k>0\).
       \item[3.] \(\mathcal{L}_{\mu}(\textbf{A}_{\zeta},\bm{\lambda}_{\zeta})= \mathcal{L}_{\mu}(\textbf{A},\bm{\lambda})\), where \( \zeta\) is any permutation of \(\lbrace1,2,\cdots,m\rbrace, \textbf{A}_{\zeta}= (A_{\zeta(1)},A_{\zeta(2)},\cdots,A_{\zeta(m)})\) and \(\bm{\lambda}_{\zeta}= (\lambda_{\zeta(1)},\lambda_{\zeta(2)},\cdots,\lambda_{\zeta(m)})\).
       \item[4.] \(\mathcal{L}_{\mu}(.)\) is continuous. 
       \item[5.] \(\mathcal{L}_{\mu}(U^*\textbf{A}U,\bm{\lambda})= U^* \mathcal{L}_{\mu}(\textbf{A},\bm{\lambda}) U\) for any unitary matrix \(U\in \mathcal{M}_n\).
       \item[6.] \(\mathcal{L}^{-1}_{-\mu}(\textbf{A}^{-1},\bm{\lambda})= \mathcal{L}_{\mu}(\textbf{A},\bm{\lambda})\). 
   \end{enumerate} 
\end{proposition}
  
\begin{proposition}\label{A-H mean is accretive}
Let \(A_i\in \Pi_{n,\alpha}, i=1,2,\cdots,m,\) and let \(\mu\geq 0\). Then
\[     \mathcal{L} _{\mu}(\Re \textbf{A}, \bm{\lambda} )   \leq  \sec \gamma \, \Re \mathcal{L} _{\mu}(\textbf{A}, \bm{\lambda} )+  ( \sec \gamma -1)\mu  I. \]
\end{proposition}

\begin{proof}
Let \(\mu \geq 0\). Then 
\begin{align*}
\mathcal{L} _{\mu}(\Re \textbf{A}, \bm{\lambda} )+\mu I &=  \left ( \sum_{i=1}^{m}\lambda_{i} (\Re (A_{i}) +\mu I) \right) \sharp \left ( \sum_{i=1}^{m}\lambda_{i} (\Re (A_{i}) +\mu I)^{-1} \right)^{-1}\\
&= \Re \left ( \sum_{i=1}^{m}\lambda_{i} ( A_{i} +\mu I) \right) \sharp \left ( \sum_{i=1}^{m}\lambda_{i} \Re^{-1}( A_{i} +\mu I) \right)^{-1}  \\
&\leq \Re \left ( \sum_{i=1}^{m}\lambda_{i} ( A_{i} +\mu I) \right) \sharp \left ( \sum_{i=1}^{m} \lambda_{i} \Re(( A_{i} +\mu I)^{-1}) \right)^{-1} \; (\text{by Lemma \ref{l1}})\\
&= \Re \left ( \sum_{i=1}^{m} \lambda_{i} ( A_{i} +\mu I) \right) \sharp \left (\Re \Big(\sum_{i=1}^{m}\lambda_{i} ( A_{i} +\mu I)^{-1} \Big)\right)^{-1} \\
&\leq \sec \alpha\,  \Re \left ( \sum_{i=1}^{m}\lambda_{i} ( A_{i} +\mu I) \right) \sharp \,\Re \left ( \sum_{i=1}^{m}\lambda_{i} ( A_{i} +\mu I)^{-1} \right)^{-1} \; (\text{by Lemma \ref{l1}})\\ 
&\leq  \sec \alpha \, \Re \left ( \left ( \sum_{i=1}^{m}\lambda_{i} ( A_{i} +\mu I) \right) \sharp \left ( \sum_{i=1}^{m} \lambda_{i} ( A_{i} +\mu I)^{-1} \right)^{-1}\right)\; (\text{by Lemma \ref{l-5}}) \\
&= \sec \alpha\, (\Re \mathcal{L} _{\mu}(\textbf{A}, \bm{\lambda} )+\mu I).
\end{align*}
Hence, the proof is complete.
\end{proof}

\begin{remark}
    If \(\mu<0\),then 
    \begin{align*}
      \sec \gamma \; \Re \mathcal{L}^{-1}_{\mu}(\textbf{A}, \bm{\lambda} )&= \sec \gamma\;  \Re \mathcal{L}_{-\mu}(\textbf{A}^{-1}, \bm{\lambda} )\\
      &= \sec \gamma\;  \Re \mathcal{L}_{\nu}(\textbf{A}^{-1}, \bm{\lambda} ), \; (\nu=-\mu>0)\\
      & \geq \mathcal{L}_{\nu}(\bm{\Re}((\textbf{A}^{-1})),\bm{\lambda})+ \nu(1-\sec \gamma) I \; (\text{by Proposition \ref{A-H mean is accretive}})\\
      &\geq \mathcal{L}_{\nu}(\sec^2 \alpha \,\bm{\Re}^{-1}(\textbf{A}),\bm{\lambda})+\nu (1-\sec \gamma) I \; (\text{by Lemma \ref{l1} and Lemma \ref{prop of A-H mean} })\\
      &= \mathcal{L}_{\nu}(\left(\cos^2\alpha\; \bm{\Re}(\textbf{A})\right)^{-1},\bm{\lambda})+\nu (1-\sec \gamma) I \\
      &= \mathcal{L}^{-1}_{\mu}(\cos^2\alpha\; \bm{\Re}(\textbf{A}),\bm{\lambda})- \mu(1-\sec \gamma) I \\
      &= \sec^2 \alpha\;  \mathcal{L}^{-1}_{\omega}(\bm{\Re}(\textbf{A}),\bm{\lambda})+ \mu(\sec \gamma-1) I, \\
    \end{align*}
    where $\omega=\sec^2 \alpha\;\mu.$
\end{remark}

Proposition \ref{real part of A-H reverse ineq} below is the reverse of Proposition \ref{A-H mean is accretive}.

\begin{proposition}\label{real part of A-H reverse ineq}
Let \(A_i\in \Pi_{n,\alpha}, i=1,2,\cdots,m,\) and let \(\mu \geq 0\). Then
\[   \Re  \mathcal{L} _{\mu}( \textbf{A}, \bm{\lambda} )   \leq  \sec^{3} \gamma \,  \mathcal{L} _{\mu}(\Re \textbf{A}, \bm{\lambda} )+  ( \sec^{3} \gamma -1)\mu  I. \]
\end{proposition}
\begin{proof}
\begin{align*}
\Re\mathcal{L} _{\mu}( \textbf{A}, \bm{\lambda} )+\mu I &= \Re \left ( \left ( \sum_{i=1}^{m}\lambda_{i} (A_{i} +\mu I) \right) \sharp \left ( \sum_{i=1}^{m}\lambda_{i} ( A_{i} +\mu I)^{-1} \right)^{-1}\right)\\
&\leq \sec^{2} \gamma\, \left( \Re \left ( \sum_{i=1}^{m}\lambda_{i} ( A_{i} +\mu I) \right) \sharp \Re \left ( \left ( \sum_{i=1}^{m} \lambda_{i} ( A_{i} +\mu I)^{-1} \right)^{-1}\right) \right) \; (\text{by Lemma \ref{l-6}})\\
&\leq \sec^{2} \gamma \, \left ( \sum_{i=1}^{m} \lambda_{i} (\Re A_{i} +\mu I) \right) \sharp \left (  \Re \left ( \sum_{i=1}^{m}\lambda_{i} ( A_{i} +\mu I)^{-1} \right)\right)^{-1} \; (\text{by Lemma \ref{l1}})\\
&= \sec^{2} \gamma  \left ( \sum_{i=1}^{m}\lambda_{i} (\Re A_{i} +\mu I) \right) \sharp \left (    \sum_{i=1}^{m}\lambda_{i}\Re \left( ( A_{i} +\mu I)^{-1} \right)\right)^{-1}\\
&\leq \sec^{2} \gamma  \left ( \sum_{i=1}^{m}\lambda_{i} (\Re A_{i} +\mu I) \right) \sharp \left (    \sum_{i=1}^{m}\lambda_{i}  \cos^{2} \gamma \, ( \Re A_{i} +\mu I)^{-1} \right)^{-1} \; (\text{by Lemma \ref{l1}})\\
&= \sec^{3} \gamma  \left ( \sum_{i=1}^{m}\lambda_{i} (\Re A_{i} +\mu I) \right) \sharp \left (    \sum_{i=1}^{m}\lambda_{i}  ( \Re A_{i} +\mu I)^{-1} \right)^{-1}\\
&= \sec^{3} \gamma  \left( \mathcal{L} _{\mu}( \Re\textbf{A}, \bm{\lambda} )+\mu I \right).  
\end{align*}

Hence, the proof is complete.
\end{proof}

The following theorem provides the accretive version of Theorem \ref{A-H mean and Riccati's Lemma}.
 \begin{theorem}\label{A-H mean and Drury's Lemma}
     \(\mathcal{L}_{\mu}(\textbf{A},\bm{\lambda})+\mu I\) is unique solution of the equation 
     \begin{equation}\label{genealization of Ricccati lemma}
\displaystyle \sum_{i=1}^{m}\, \lambda_i\, X(A_i+\mu I)^{-1} X= \sum_{i=1}^{m}\,\lambda_i (A_i+\mu I).        
     \end{equation}
 \end{theorem}
 \begin{proof}
     First of all, notice that using the definition of the geometric mean, one can simply show that  \(\mathcal{L}_{\mu}(\textbf{A},\bm{\lambda})+\mu I\) satisfies equation \eqref{genealization of Ricccati lemma}.\\
     On the other hand, if \(X\in \Gamma_n\) satisfies equation \eqref{genealization of Ricccati lemma}, then 
     \begin{align*}
     \sum_{i=1}^{m}\,\lambda_i (A_i+\mu I)& =    \displaystyle \sum_{i=1}^{m}\, \lambda_i\, X(A_i+\mu I)^{-1} X\\
     &= X \Big( \sum_{i=1}^{m}\, \lambda_i\, (A_i+\mu I)^{-1} \Big) X.\\
     \end{align*}
  So, by Drury's Lemma, we have 
  \[X =\Big( \sum_{i=1}^{m}\, \lambda_i\, (A_i+\mu I)^{-1} \Big)^{-1} \sharp  \Big(\sum_{i=1}^{m}\,\lambda_i (A_i+\mu I)  \Big),\] 
  which, by the commutativity of the geometric mean, is nothing but \(\mathcal{L}_{\mu}(\textbf{A},\bm{\lambda})+\mu I\).
  \end{proof}
In \cite{kim}, it was shown that by taking appropriate limits, the parametrized weighted \(\mathcal{A}\sharp \mathcal{H}\) approaches the weighted arithmetic mean and the weighted harmonic mean. The proof uses delicate treatment of a certain function. With the help of Drury's Lemma, we can adopt the same proof as follows.
 \begin{theorem}
     \(\displaystyle \lim_{\mu \rightarrow \infty} \mathcal{L}_{\mu}(\textbf{A},\bm{\lambda})= \mathcal{A}_{\bm{\lambda}}(\textbf{A})\) and \(\displaystyle \lim_{\mu \rightarrow -\infty} \mathcal{L}_{\mu}(\textbf{A},\bm{\lambda})= \mathcal{H}_{\bm{\lambda}}(\textbf{A})\).
 \end{theorem}
 \begin{proof}
Let \(\mu >0\) and set \(g(t) =\sum_{i=1}^{m}\, \lambda_i\, (t A_i+ I) \sharp \sum_{i=1}^{m}\, \lambda_i\, (tA_i+ I)^{-1} \Big)^{-1}\). Clearly, \(g(0)=I\). \\
Notice that by Drury's Lemma, we have \[g(t)\Big(\sum_{i=1}^{m}\, \lambda_i\, (t A_i+ I)^{-1} \Big) g(t)= \sum_{i=1}^{m}\, \lambda_i\, (t A_i+ I). \]
Since \(g\) is defined on a neighborhood of \(0\) and is differentiable at \(0\), differentiating both sides with respect to \(t\), and setting \(t=0\) yields
\[g^{\prime}(0) - \sum_{i=1}^{m}\, \bm{\lambda}_i A_i +g^{\prime}(0)= \sum_{i=1}^{m}\, \lambda_i A_i \,.\]
Consequently, we get
\begin{equation}\label{Arithmatic mean as the derivative of g}
\mathcal{A}_{\bm{\lambda}}(\textbf{A})= \sum_{i=1}^{m}\, \lambda_i A_i = g^{\prime}(0)= \displaystyle \lim_{t \rightarrow 0^+} \,\dfrac{g(t)-g(0)}{t}\,.
\end{equation}
Now,
\begin{equation*}
\begin{split}
    \mathcal{L}_{\mu}(\textbf{A},\bm{\lambda})& = \Big(\mathcal{A}_{\mu,\bm{\lambda}}(\textbf{A}) \sharp \mathcal{H}_{\mu \bm{\lambda}}(A)\Big)-\mu I\\
    &= \mu \Big( \left(\sum_{i=1}^{m}\, \lambda_i\, (\mu^{-1} A_i+ I) \sharp \Big(\sum_{i=1}^{m}\, \lambda_i\, (\mu^{-1}A_i+ I)^{-1} \Big)^{-1}\right)-I\Big)\\
    &= \Big(\dfrac{g(t)-g(0)}{t}; \quad t=1/\mu\Big).
\end{split}
\end{equation*}
Thus, by \eqref{Arithmatic mean as the derivative of g}, we get
\[\mathcal{A}_{\bm{\lambda}}(\textbf{A})=\displaystyle \lim_{t \rightarrow 0^+} \,\dfrac{g(t)-g(0)}{t}= \displaystyle \lim_{\mu \rightarrow \infty}\, \mathcal{L}_{\mu,\bm{\lambda}}(\textbf{A}).\]

Now, from \(\mathcal{L}_{\mu}(\textbf{A})= \mathcal{L}^{-1}_{-\mu}(\textbf{A}^{-1})\) if \(\mu<0\), we get 
\[\displaystyle \lim_{\mu\rightarrow -\infty}\, \mathcal{L}_{\mu}(\textbf{A})= \displaystyle \lim_{-\mu\rightarrow \infty}\, \mathcal{L}_{-\mu}(\textbf{A})= \displaystyle \lim_{\mu\rightarrow \infty}\, \mathcal{L}^{-1}_{\mu}(\textbf{A}^{-1})= \mathcal{A}^{-1}_{\bm{\lambda}}(\textbf{A}^{-1})= \mathcal{H}_{\bm{\lambda}}(\textbf{A}).\]
 \end{proof}

\section*{Acknowledgements}
Not applicable.


\end{document}